%% file: main.tex
\documentclass{article}

\usepackage[preprint]{neurips_2025}

\usepackage[english]{babel}
\usepackage{amssymb,amsthm,amsfonts}
\usepackage[utf8]{inputenc} 
\usepackage[T1]{fontenc}    
\usepackage{amsmath,amsthm}
\usepackage{graphicx}
\usepackage{todonotes}
\usepackage{hyperref}       
\usepackage{url}            
\usepackage{booktabs}       
\usepackage{nicefrac}       
\usepackage{microtype}      
\usepackage{xcolor}         
\usepackage[normalem]{ulem}

\usepackage{tikz}
\usetikzlibrary{calc, intersections}
\usetikzlibrary{arrows, shapes, fit}
\usetikzlibrary{positioning}
\usetikzlibrary{arrows.meta}
\usetikzlibrary{matrix,decorations.pathreplacing, calc, positioning,fit}
\usepackage{tkz-graph}
\usetikzlibrary{patterns,snakes}
\usetikzlibrary{arrows, shapes, fit}
\usetikzlibrary{shapes.geometric,calc}
\usepackage{array}
\usepackage{pgfplots}

\usetikzlibrary{arrows.meta}
\tikzset{main node/.style={rectangle,fill=blue!20,draw=none,minimum size=0.3cm,inner sep=3pt, rounded corners = 3pt,font=\sffamily},}

\newcommand{\fat}{\text{fat}}

\input{prem.tex}

\title{How hard is learning to cut? \\ Trade-offs  and sample complexity}

\author{
  Sammy Khalife \\
  School of Operations Research and Information Engineering\\
  Cornell Tech, Cornell University\\
  \texttt{khalife.sammy@cornell.edu} \\
  \And
  Andrea Lodi \\
  Jacobs Technion-Cornell Institute \\
Cornell Tech and Technion - IIT \\ 
  \texttt{andrea.lodi@cornell.edu} \\
}

\begin{document}


\maketitle

\begin{abstract}
In the recent years, branch-and-cut algorithms  have been the target of data-driven approaches designed  to enhance the decision making in different phases of the algorithm such as branching, or the choice of cutting planes (cuts). In particular, for cutting plane selection two score functions have been proposed in the literature to evaluate the quality of a cut: branch-and-cut tree size and gap closed.  In this paper, we present new sample complexity lower bounds, valid for both scores.  
We show that for a wide family of classes $\mathcal{F}$ that maps an instance to a cut, learning over an unknown distribution of the instances to minimize those scores requires at least (up to multiplicative constants) as many samples as learning from the same class function $\mathcal{F}$ any generic target function (using square loss). 
Our results also extend to the case of learning from a restricted set of cuts, namely those from the Simplex tableau. To the best of our knowledge, these  constitute the first lower bounds for the learning-to-cut framework. We compare our bounds to known upper bounds in the case of neural networks and show they are nearly tight.  We illustrate our results with a graph neural network selection evaluated on set covering and facility location integer programming models and we empirically show that the gap closed score is an effective proxy to minimize the branch-and-cut tree size. Although the gap closed score has been extensively used in the integer programming literature, this is the first principled analysis discussing both scores at the same time both theoretically and computationally.
\end{abstract}

\input{introduction}

\section{Preliminaries}
\label{sec:preliminaries}

In this section, we provide preliminaries for both ILP cutting plane methodology and learning theory.

\subsection{Branch and cut and cutting planes}
\label{section:branchandcut}

We consider the ILP in the form
    \begin{equation}
        \label{def:ILP}
        \max\{\c^\T \x : A \x \leq \b, \x \geq 0, \x \in \Z^n\},
    \end{equation}
where $m,n \in \N_+$, and $A \in \Q^{m \times n},\ \b \in \Q^m,\ \c \in \R^n$.\footnote{The ILP \eqref{def:ILP} is called MILP if a subset of the variables is allowed to take continuous values.}

The algorithms implemented in every (M)ILP solver are variations of a framework called {\em branch and cut}. 
In that algorithm, each iteration maintains: 1) a current best (integral) solution guess,\footnote{Such guess would likely be $-\infty$ initially.} and 2) a list of polyhedra, each a subset of the original ILP relaxation. At each step, one polyhedron is selected and its continuous LP solution is computed. If the objective is worse than the current guess, the polyhedron is discarded. If the solution is integral, the guess is updated and the polyhedron is removed. Otherwise, the algorithm either adds \emph{cutting planes} -- valid inequalities that tighten the polyhedron -- or \emph{branches}. In branching, a variable $\x_i$ whose current value $\x^*_i$ is  fractional is chosen, and the polyhedron is split using $\x_i \leq \lfloor \x^*_i \rfloor$ and $\x_i \geq \lfloor \x^*_i \rfloor + 1$. These two new polyhedra replace the original one. This process builds a branch-and-cut tree, with each node representing a polyhedron. The algorithm stops when the list is empty, returning the best guess as optimal. Often, a bound $B$ is set on the tree size; if exceeded, the algorithm terminates early and returns the current best guess.

There are many different strategies to generate cutting planes in branch-and-cut~\cite{conforti2014integer,nemhauser1988integer,sch}. The oldest one is due to Gomory \cite{Gomory58} and later generalized by Chv\'atal \cite{Chvatal73}, so the family of resulting cutting planes is called Chv\'atal-Gomory cuts. Namely, for any $\x \in \Z^n$ satisfying $A \x \le \b$, then the inequality $\u A \x \le \lfloor \u \b \rfloor$ is valid for $S$ for all $\u \ge \0$ such that $\u A \in \Z^n$ and is called a CG cut. Gomory suggested to read $\u$ as the inverse of the basis of the tableau when the LP relaxation is solved by the Simplex method \cite{Gomory58}. Chv\'atal generalized the procedure to any $\u$ \cite{Chvatal73}.

Since the number of CG cuts that can derived at any iteration of the branch-and-cut algorithm is very large, any MILP solver implements its own cut selection strategy, i.e., decides which cuts are added to the current LP relaxation. The cut selection is performed by sophisticated, handcrafted heuristics and, as anticipated, the use of modern statistical learning to enhance these heuristics has been recently studied. The standard approach that has been used and that we inherit here is to decide the \emph{single} next cut to be added within the CG family (or part of it). To do so, we need a score function that evaluates the quality of the cut, and two such functions have been investigated. Ideally, the branch-and-cut tree size \emph{after} the addition of the cut is the right measure since most of the computing time is spent on solving the individual LPs in the nodes of the algorithm. However, this scoring function is very expensive to evaluate and, so far, has been used for theoretical purposes only. Instead, MILP technology generally measures the quality of a cut using the gap closed, i.e., the measure of the improvement of the LP relaxation after the addition of the cut. Of course, this is cheaper to evaluate (requires to solve one single LP per cut), but still too expensive in practice for performing cut selection, so the idea of \emph{learning} such a score.\footnote{It is worth mentioning that no solver adds one cutting plane at a time, but cuts are instead added in groups, called rounds. Analyzing such a procedure would be way harder, so literature studies -- as well as our paper -- concentrate on this simplified version.}

It is interesting to note that, although the gap closed could be seen as a proxy of the branch-and-cut tree size, the two scores are hard to properly compare. More precisely, a cut could reduce significantly the tree size without even cutting off the optimal (fractional) solution of the LP relaxation, while a cut that does cut it off could have no effect long term, i.e., in reducing the tree size. 

For example, consider the ILP 
$\{\max  5x_1 + 8x_2~|~x_1 + x_2 \leq 6, 5x_1 + 9x_2 \leq 45, x_1, x_2 \geq 0, x_1, x_2 \in \Z\}$, 
whose fractional solution is $x^{*} = (\frac{9}{4}, \frac{15}{4})$. It can be shown that one of the CG cuts derived from the optimal tableau leads to the constraint $4 x_1 + 7 x_2 \leq 35$. Adding this constraint leads to a new fractional solution $ (\frac{7}{3}, \frac{11}{3})$, located on the right (i.e., with greater $x$-coordinate) of the solution of the original formulation. Hence, supposing branching is performed first on $x_1$ then $x_2$, this leads to a larger branch-and-cut tree, with more LPs to be solved. However, this CG cut actually cuts off the fractional solution, hence improves the gap closed score.


\subsection{Learning theory}

\begin{definition}[Restriction of a concept class]\label{def:restriction}
Let $\mathcal{F}$ be a concept class  (i.e., set of functions) from $\R^{d} \rightarrow \R$. For any $i \in [d]$ and $c \in \R^{d-i}$, we refer to $\mathcal{F}_{i}(c)$ as a shorthand for 
$$\mathcal{F}[i](c) := \{ x \mapsto f(x_1, \cdots, x_i, c): f \in \mathcal{F}\}$$  
\end{definition}
 
\begin{definition}[VC-dimension of a real output concept class]\label{def:pseudo-dimension} 
    For any positive integer $t$, we say that a set $\{I_1, ..., I_t\} \subseteq \I$ is shattered by a concept class  $\mathcal{E}$ defined on $\mathcal{I}$ taking $\{0,1\}$-values if 
    \begin{equation*}\label{eq:def:pseudodimension}
        2^t = \left|\left\{ \left( f(I_1),\dots, f(I_t) \right) : f \in \mathcal{E}\right\}\right|
    \end{equation*}
    
    The \emph{VC dimension of $\mathcal{E}$}, denoted as $\VCdim(\mathcal{E})  \in \mathbb{N} \cup \{+\infty\}$, is the size of the largest set that can be shattered by $\mathcal{E}$. 
    
     If $\F$ is a non-empty collection of functions from an input space $\I$ to $\R$. Let $\sgn(\mathcal{F}):=\{\sgn(f) \in \mathcal{F}\}$ where $\sgn(x)=\mathbf{1}_{x > 0}$. Then, $\VCdim(\mathcal{F})$ is by definition $\VCdim(\sgn(\mathcal{F}))$ where we adopt the standard defintion of $\VCdim $ for $\{0,1\}$-function described above.
    \end{definition}

\begin{definition}[Fat-shattering dimension] Let $\gamma > 0$. With the same notations as Definition \ref{eq:def:pseudodimension}, we say that the function class $\mathcal{F}$ fat-shatters $I_1, \cdots, I_t$ with precision $\gamma$ provided there exists $r \in \R^{t}  $ such that for every labeling $(y_1, \cdots, y_t) \in \{-1, 1\}^{t}$, there exists $g \in \mathcal{F}$, such that $g(I_i) \geq r_i + \gamma$ if $y_i = -1$ and $g(I_i) \leq r_i - \gamma$ if $y_i =1$. In such conditions, $r$ is called the witness of the shattering. The fat-shattering dimension of $\mathcal{F}$ with precision $\gamma$, noted $\fat_{\F}(\gamma)$ is the size of the largest  that can be fat-shattered by $\mathcal{F}$.
\end{definition}

We are interested in a  statistical learning  problem of the following form, given a fixed parameterized function class defined by some $h$ with output space $\O = \R$:
\begin{equation}\label{eq:learn}
    \min_{f \in \F}\; \mathbb{E}_{I \sim \mathcal{D}}[h(I, f(I))],
\end{equation} 
for an unknown distribution $\mathcal{D}$, given access to i.i.d. samples $I_1, \ldots, I_t$ from $\mathcal{D}$. We restrict to learning problems of a function $f$ to minimize a given functional measuring the quality of a cutting plane in a branch-and-cut type of algorithm. 
In this problem, one tries to learn the best decision $f \in \mathcal{F}$ for minimizing an expected ``score'' with respect to an unknown distribution given only samples from the distribution. In our branch-and-cut framework, we assume that we have access to an oracle returning the performance of the cutting plane after adding it to the ILP instance, that will be accounted for in the choice of the function $h$. We are interested in two performance scores: one related to the relative variation of the size of the branch-and-tree after adding the cut, and the gap closed score. Both will be formally defined in Section \ref{section:branchandcut}. Note that Formulation \ref{eq:learn} is unsupervised in the sense that we are not trying to perform standard regression to some observed values. 
However, we can reduce to the supervised learning framework developped in \cite[Section 9]{anthony2009neural} by considering only 0 labels and using the square loss function. Therefore, all known results on sample complexity of learning can transfer to this setup.  

In this context,  a \emph{learning algorithm}\footnote{In supervised learning, typically the domain is formed by $\cup_{m=1}^{\infty}(\mathcal{I} \times \R)^{m}$.} $L$ for $\mathcal{F}$  is a function taking as input a fixed simple sample of arbitrary size, and returning a function in $\mathcal{F}$:
$$ L: \bigcup_{m=1}^{\infty}\mathcal{I}^{m} \rightarrow \mathcal{F}$$

Given $\epsilon \in (0,1)$, $\delta \in (0,1)$, the \emph{sample complexity of learning} $m_{0} (\epsilon, \delta) \in \N$ of $L$ is the smallest  integer (allowed to be $+\infty $) such that for any $m \geq m_{0} (\epsilon, \delta)$, for any probability distribution $\mathcal{D}$ on $\mathcal{I}$, the algorithm $L$ evaluated at ``test time'' on instance $I$ is is in average closs to the solution on the entire distribution up to $\epsilon$:
$$\left| \mathbb{E}_{I\sim \mathcal{D}}[ h\left(I, L(I_1, \cdots, I_m)(I)\right) ]- \min_{f \in \F}\; \mathbb{E}_{I \sim \mathcal{D}}[h(I,f(I))] \right| < \epsilon $$ with probability $1-\delta$ over i.i.d  samples $I_1, \cdots, I_m$ drawn following $\mathcal{D}$.



    In the case of binary functions, VC-dimension  gives a direct way to bound \emph{from above and below} learning sample complexity \cite[Theorem 5.4]{anthony2009neural}.
    For real output functions, the pseudo-dimension remain useful to upper-bound on \emph{uniform convergence} (UC). UC typically requires the difference $\left|\frac{1}{m}\sum_{i=1}^m h(I_i,f(I_i)) - \mathbb{E}_{I \sim \mathcal{D}}\;[h(I,f(I))]\right| $ to be bounded by  $\epsilon $ for every $f \in \mathcal{F}$ and for every distribution. However, the sample complexity of learning can be smaller than that of UC. This leads to the sample complexity of UC to be an upper-bound on the sample complexity of learning via Empirical Risk Minimization (ERM), which is itself greater than the sample complexity of learning in general, as there could be other algorithms performing better than ERM.
    In other words, uniform convergence guarantees that ERM will perform well, since the sample average closely matches the true expectation across all hypotheses. Good performance from ERM can still occur without full uniform convergence, and there may exist other learning algorithms that outperform ERM. 
    
    Therefore, lower-bounds on Pseudo-dimension or VC-dimensions mainly apply to UC, and do not necessarily reflect the true sample complexity of learning. This surprising gap was first highlighted in \cite{shalev2009stochastic}  and further explored in \cite{feldman2016generalization}. As a consequence, to obtain lower-bounds of learning sample complexity, one cannot \emph{a priori} use  standard traditional lower-bounds of VC-dimension, and the analysis has to be performed carefully depending on the concept class considered. In this article, we will rely on the following result giving  a general lower-bound on the sample complexity of learning.

\begin{theorem}\cite[Theorem 19.5]{anthony2009neural}\label{theorem:lowerbound:samplecomplexity}
 Let $\mathcal{F}$ be a class of functions from $X$ to $[0,1]$. Then for any  $0 < \epsilon < 1$, $0 < \delta < 10^{-2}$, any learning algorithm $L$ for $\mathcal{F}$ has sample complexity  $m_L(\epsilon, \delta) $ satisfying for every $0 < \alpha < \frac{1}{4}$,
$$m_L(\epsilon, \delta) \geq \frac{\fat_{\mathcal{F}}(\frac{\epsilon}{\alpha})  - 1 }{16 \alpha}   $$
\end{theorem}
Thus, any learning algorithm will have to use at least $\frac{\fat_{\mathcal{F}}(\frac{\epsilon}{\alpha})  - 1 }{16 \alpha} $ samples to guarantee that the average solution at test time, independently of the distribution, will be at most at $\epsilon$ distance from the best solution of the function class, with probability $1-\delta$. Note that the lower-bound is rigorously valid only when $\delta < \frac{1}{100}$ (and the bound becomes  independent of $\delta$ in that regime). 

\begin{remark}
    Sample complexity results are usually presented in the supervised setting and depend on the choice of the loss function, otherwise bounds have to be slightly modified \cite{anthony2009neural}. In order to avoid this discussion in the branch-and-cut framework, we can use the following trick for theoretical purposes: we choose all labels to be equal to $0$ and use a score function only taking values in $[0,1]$, in order to consider the equivalent problem of  minimizing the average square loss of the score, instead of the score itself.
\end{remark}



\section{Statement of results}
\label{sec:theory}

For any positive integer $ d \in \Z_+ $,   $ [d] $ refers to  the set $ \{1, 2, \ldots, d\} $ . 
The sign function  $\sgn: \R \rightarrow \{0,1\}$, is defined such that for any $x \in \R$, $\sgn(x) = 0$ if $x < 0$, and $1$ otherwise. This function is applied to each entry individually when applied to a vector. 
The elementwise floor function $\lfloor \cdot \rfloor$ is used to indicate  the rounding down of each component of a vector to the nearest integer.


\subsection{Over the pool of all CG-cuts}\label{subsection:results1}

We first present results in the case where the generation of CG-cuts is unrestricted, i.e., except the limitations brought by the generation process using the concept class, the whole pool of CG-cuts is considered. We assume the following structure on the underlying concept class $\mathcal{F}$ used to generate the CG-cuts: each function of the concept class incorporates an encoder function to transform each ILP to be processed further. For Neural networks, an example of such an encoder is the concatenation of all the instance's numerical data into a single vector. In the case of Graph Neural Networks, one can choose graph based representation (cf. for example  \cite{chen2024rethinking}). For ease of presentation, we will in the following suppose that we are working with the stacking encoder, and functions of $\mathcal{F}$ have domain $\R^{n \times m + m + n}$ and codomain $ \R^{m}$ where $n$ is the number of variables of the ILP, and $m$ its number of constraints. Our Assumptions described can be generalized to other concept classes  including the \hypertarget{assumption1}{GNN ones}.

\begin{enumerate}
\item[] \textbf{Assumption 1.}  $\mathcal{F} $ is a non empty concept class closed under translation of the input, i.e., for every $ \mu \in \R^{n \times m +m +n}$, $f \in \mathcal{F} \implies x \mapsto f ( x + \mu) \in \mathcal{F}$, and under scaling of the output of every coordinate, i.e., for every real $\lambda$ and $i \in [m]$, and $f =(f_1, \cdots, f_m) \in \mathcal{F}$ implies that $(f_1, \cdots, \lambda f_i, \cdots, f_m) \in \mathcal{F}$. Note that is true for (graph) neural networks (for any activation function that is not identically zero).
\item[] \textbf{Assumption 2.} (Same shattering power by restriction to some row).  Let $r = m \times n + m + n$. 
For every $i \in [m]$ representing the index of the associated CG-weight,  $c \mapsto \VCdim(\mathcal{F}_i[n](c)) $ is constant (cf. Definition \ref{def:restriction}, here $\mathcal{F}_i$ refers to the concept class formed by the $i-$ coordinate of $f \in \mathcal{F}$). This is for example true for (graph) Neural Networks with any activation function\footnote{This can be seen by adjusting the bias of the Neurons in the first layer.}. In those conditions, we refer to this constant as $\VCdim(\mathcal{F}_n)$.
\end{enumerate}





\begin{definition}
    Let $s: \mathcal{I} \times [0,1]^{m} \rightarrow \R$ be a score function, mapping each pair formed by an ILP instance and a weight vector of a CG cut to a real value.
Let $\mathcal{F}$ be a concept class following assumptions described above. Let $\sigma' : \R^{m} \rightarrow [0,1]^m$  be a \emph{squeezing function} so that   $ \sigma' \circ f $ (where $f \in \mathcal{F}$)  returns a vector in $[0-1]^{m}$ used as weights of the CG-cuts. We also suppose that $\sigma'$ verifies $\sigma'((-\infty, 0)) \subset [0, \frac{1}{2})$, $\sigma'([0, +\infty)) \subset [\frac{1}{2}, 1]$ and $(0,1) \subset \sigma'(\R)$. Let  $\mathcal{F}_{\sigma'}$ be the concept class obtained.  We define  $\mathcal{F}_{s,\sigma'}$ as the final resulting concept class $$\mathcal{F}_{s,\sigma'} := \{ I \mapsto s(I, h(I)) : h \in \mathcal{F}_{\sigma'} \}$$
\end{definition}

\begin{theorem}\label{theorem:lowerboundtransfer}
Under those assumptions, for both gap-closed and branch-and-cut tree size scores, the sample complexity of learning CG-cuts via the class $\mathcal{F}_{s, \sigma'}$ verifies 
$$      m_L(\epsilon, \delta) = \Omega ( \frac{ \VCdim( \mathcal{F}[n])  }{  \epsilon} )  $$
where $\VCdim( \mathcal{F}[n]) := \max_{i\in[m]} \VCdim( \mathcal{F}_i[n]) $.
\end{theorem}

According to the notion of  learnability, Theorem \ref{theorem:lowerboundtransfer} provides a lower bound on the minimum number of samples required to guarantee with probability $1-\delta$ that for any distribution $\mathcal{D}$, the solution of \emph{any learning algorithm} (in particular, this is true for the Empirical Risk Minimizer (ERM) algorithm) returns a solution whose predictions are at most $\epsilon$ far from the optimal neural network with high probability over the entire distribution.  

\begin{corollary}\label{corollary:learningcomplexity}
    For any concept classes verifying Assumptions \hyperlink{assumption1}{1 and 2}, 
    $ m_{L}(\epsilon, \delta) $ is bounded from below by the sample complexity of learning from $\mathcal{F}_n$ to a generic target function. In particular, for every $\gamma > 0$, we have with the same notations as Theorem \ref{theorem:lowerboundtransfer},
    $$ m_{L}(\epsilon, \delta)  = \Omega \left( \frac{\fat_{\mathcal{F}[n]}(\gamma)}{\epsilon}\right) = \Omega \left( \frac{\VCdim(\mathcal{F}[n])}{\epsilon}\right)$$
    where similarly $\fat_{\mathcal{F}[n]}(\gamma) := \max_{i\in[m]} \fat_{\mathcal{F}_i[n]}(\gamma) $.
\end{corollary}
 Corollary \ref{corollary:learningcomplexity}  applies in particular to neural networks (and to graph neural networks as well), up to adding an extra neuron on each layer\footnote{There is no asymptotic difference  between Pseudo-dimension and VC-dimension of real output neural networks, up to adding one layer or one neuron per layer.}.

We now compare to the known upper-bound in the case of Neural Networks (i.e., when $\mathcal{F}$ is composed of Neural Networks of a certain depth and width). The upper bound of the pseudo-dimension of this concept class given by \cite[Proposition 3.3]{cheng2024data} is $ \mathcal{O}\left( LW \log (U+m) + W \log  M  \right)$ for ReLU neural networks and a squeezing function to constrain their outputs in $[0,1]$, $M $ is an upperbound on the coefficients in $A$ and $b$,  where $U$ is the \textit{size} of the neural network, defined as $w_1 + \cdots + w_W$, and  are also imposed the conditions that $\sum_{i=1}^m \sum_{j=1}^n |A_{ij}| \leq a$ and $\sum_{i=1}^m |b_i|  \leq b$ for any $(A,\b,\c) \in \I$, and $M := 2(a + b + n)$. 

Hence, ignoring logarithmic factors in $\frac{1}{\delta}$ and $\frac{1}{\epsilon}$, the best known upper-bounds for $m(\epsilon, \delta)$ is given by $\O\left( \frac{1}{\epsilon^2}\left(  LW \log (U+m) + W \log M \right) \right) $, for the BC tree size score. Since the result only use the invariance by the number of regions where the CG-cuts remain constants, their proof  can  adapted for the gap closed score, although we suspect that a better upper bound should be achievable in that case.

 We now state our lower-bound in the case of Neural Networks in the next proposition. Note that our lower-bound does not use any amplitude on the input data of the problem.

\begin{proposition}\label{lowerbound:neuralnetworks}
   Suppose $\mathcal{F}$ is composed of ReLU neural networks with $\leq L$ layers, and $ \leq W $ weights,  with the concatenation encoder $I \in \mathcal{I} \mapsto (A, b, c) \in \R^{n \times m + m + n}$. There is a universal constant $C$ such that the following holds. Suppose $W > CL  > C^2$   Consider both gap-closed and branch-and-cut tree size scores. Then, the sample complexity of learning CG-cuts via the class $\mathcal{F}_{s, \sigma'}$ verifies 
$$    m_L(\epsilon, \delta) \geq \frac{1}{  \epsilon C}  \overline{W}   L \log\left(\frac{\overline{W}}{L}\right)  
$$   
where $\overline{W} := W - w_1(n+1)m  $.
\end{proposition}
A few comments are in order:
\begin{itemize}
\item The correction term of $w_1(n+1)m$, where $w_1$ is the number of neurons in the first layer, accounts for the restriction of the concept class to $n$ inputs, $\mathcal{F}[n]$. Our approach ``ignores''  $n\times m +m = (n+1)m$ inputs. This leads to an amount of $w_1(n+1)m$ weights that are being removed in the neural network. 
    \item Recall that $W = \sum_{i=1}^{L}w_{i-1} w_i$
where $w_0:= n \times m + m  + n$ is the input dimension, and  the other $w_i$'s are the widths (number of neurons) of the Neural network considered on each layer. In particular $\overline{W} = W - w_1(n+1)m$ is always positive, and furthermore the ratio $\frac{\overline{W}}{W}  
$ is greater than $  1 - \frac{w_1w_0}{1 + W} \geq 1 - \frac{W}{ 1+W}$. 

\item The lower bound suppose some structure on the layers and parameters given by $W > CL $. This loss of generality does not take place in our proof technique, but in the bit-extration technique to lower-bound the $\VCdim$ of the class of neural networks \cite{bartlett2019nearly}. Therefore, in order to remove that assumption, one would have the either to obtain lower-bound that do not require that structure, or adopt an entirely different approach, specific to shattering ILPs,  that would not require a general VC dimension lower bound on Neural Networks.
\end{itemize}

 Hence, supposing a regime where the number of weights in the Neural Network are large compared to the variables $n$ and number of constraints $m$,  the  gap of is of order $\frac{1}{\epsilon}$, between our lower-bound and the best upper bound, ignoring logarithmic factors in $\frac{1}{\delta}$ and $\frac{1}{\epsilon}$. In a general learning framework, this gap is inevitable: see for instances discussions in \cite[Section 19.5]{anthony2009neural}. We suspect that this gap transfers for learning CG cuts, if  no further assumption is made on the distribution of instances.

\subsection{Over the pool of all CG-cuts from the tableau}

We now restrict to the pool of CG-cuts obtained from the tableau, so the concept class has to be changed slightly. We show that despite our restriction, the sample complexity is still driven by the VC-dimension of the underlying concept class.


To make this formal, we suppose the following structure: each function of the concept class is decomposable as the composition of 
\begin{itemize}
    \item a  function that  takes as input an ILP instance $I\in \mathcal{I}$ and returns the $m$ CG-cuts from the tableau  $(a_1,b_1), \cdots, (a_m, b_m)$. This can be perfomed using the the simplex algorithm.
\item Each function $g \in \mathcal{G}$ maps $(I, a_i,b_i)$ to a real value. The cut selected to be added to the instance is the one maximizing each of the $m$ scores, the concept class after selecting the maximum is $\tilde{G}$ (ties are broken by alphabetical order of the constraints). 
\end{itemize}




\begin{definition}
    Let $s: \mathcal{I} \times [0,1]^{m} \rightarrow \R$ be a score function, mapping each pair formed by an ILP instance and a weight vector of a CG cut to a real value.
Let $\mathcal{G}$ be a concept class described above such that Assumptions \hyperlink{assumption1}{1 and 2} hold.
We define $\mathcal{G}_{s}$ as the final resulting concept class $$\mathcal{G}_{s} := \{ I \mapsto s(I, g(I)) : g \in \tilde{\mathcal{G}} \}$$
\end{definition}

\begin{theorem}\label{theorem:lowerboundtransfertableau}
Under those conditions, for both gap-closed and branch-and-cut tree size scores, the sample complexity of learning CG-cuts via the class $\mathcal{G}_{s, \sigma'}$ verifies 
$$      m_L(\epsilon, \delta) = \Omega ( \frac{\VCdim(\mathcal{G}[n])  }{  \epsilon} )  $$
\end{theorem}

\begin{proposition}\label{theorem:lowerboundtransfertableau}
 Suppose $\mathcal{G}$ is composed of neural networks with $\leq L$ layers, and $ \leq W $ layers,  with the concatenation encoder $I \in \mathcal{I} \rightarrow (A, b, c) \in \R^{n \times m + m + n}$. There is a universal constant $C$ such that the following holds. Suppose $W > CL  > C^2$.  Then the sample complexity of learning CG-cuts via the class $\mathcal{F}_{s}$ from the optimal Tableau verifies  
$$    m_L(\epsilon, \delta) \geq \frac{1}{  \epsilon C}  \overline{W}   L \log\left(\frac{\overline{W}}{L}\right)  
$$   
where $\overline{W} := W - w_1 (n+1)(m+1)    $.
\end{proposition}

In comparison with the upper-bounds \cite[Corollary 2.8]{cheng2024data}, ignoring logarithmic factors in $\frac{1}{\delta}$ and $\frac{1}{\epsilon}$,  we have that $ m(\delta, \epsilon) = \O\left( \frac{WL\log(Um)}{\epsilon^2} \right)$.
where $U = w_1 + \cdots w_L$ is the total number of neurons.  Therefore, seen as a function of the parameters, supposing the regime where the number of weights in the Neural Network are large compared to the variables $n$ and number of constraints $m$, which implies $\overline{W} $ to be of the order of $W$, our bound could be improved by integrating   logarithmic factors in $m$ and $U$. 

\section{Numerical experiments}
\label{sec:numerical}


To start the sample complexity analysis computationally, we wish to investigate how both scores in the literature relate empirically, based on the premise that (i) the reduction in the B\&C tree size is ultimately the score of interest but is costly to obtain and learn, and (ii) the gap closed is easier to compute but less reliable as a training signal for the end-task of minimizing the B\&C tree size.

A potential trade-off emerges: cuts that close large gaps may not always reduce tree size due to some situations where both are incomparable (see end of Section \ref{section:branchandcut} ), or from branching decisions that change the impact of one cut overall. This setup mirrors classic proxy optimization challenges in machine learning, where we want to learn for a costly target (tree size), but we use a cheaper, noisier proxy (gap closed), hoping for performance generalization to the target.

Our computational methodology is based on the two key building blocks: (1) We represent each pair (ILP instance, cut) as a graph, i.e., we encode variables and constraints by a graph neural network (GNN), with proper edges  and features. GNNs naturally encode ILP instances well because the solution of an ILP does not depend on the order of the rows, which is captured by the isomorphism invariance of the associated representation. (2) At training time, we generate all the CG cuts from the optimal Simplex tableau with corresponding scores, for any considered ILP instance. The GNN is trained to match the scores returned for each CG cut via a cross-entropy loss. 

Having collected the (up to) $m$ CG cuts from the optimal Simplex tableau, and their corresponding scores $s_1, \dots, s_m$ (either gap closed or B\&C tree size reduction), for each of the $t$ instances, we approximate a solution of the problem
 \begin{equation}
    \min_{W} \frac{1}{t} \sum_{i=1}^{t} \ell\big(  (H_W ( E (I_i, o_i)) )_{j \in [m]}, (s_j)_{j \in [m]} \big), 
\end{equation}
where $E$ is the instance and cut encoder (described in the next subsection), $H_W$ is a GNN parametrized by the weights $W$, which takes as input a graph  and vectors $o_i$ of size $m+1$ representing the collected cut from the tableau (left-hand side and right-hand side), and $\ell$ is the cross entropy loss $\ell: \R^{m} \times \R^{m} \rightarrow \R, (x,y) \mapsto \ell(x,y):= \frac{1}{m}\sum_{k=1}^{m} y_k \log(\frac{x_k}{\sum_{l=1}^{l} x_l})$. At inference time, suppose the trained parameters is given by $W$. On a new instance $I$, the CG cut will be selected as $\argmax_{i \in [m]} H_W(I, o_i)$, and ties are broken in an uniformly random manner.

\subsection{Experimental setup}

\paragraph{Modeling ILP as GNN.} Each ILP instance augmented by a cut gets encoded by $E$ unambiguously as a weighted graph $G$ with a three dimensional feature vector on its vertices as follows: (i) The vertices of $G$ are split between the variables and constraint vertices. Each variable gets associated to a vertex, and each constraint as well, leading to a bipartite graph with $n + m$ vertices. Furthermore, each variable vertex receives a three-dimensional feature vector corresponding to the objective vector entry, plus the coefficient of the cut for that variable, as well as the right-hand side (same for all variables). The other vertices corresponding to constraints get the vector $(1,1,1)$ as feature (for dimensional homogeneity purposes). (ii) One edge is created between each variable vertex $i$ and constraint vertex $j$ provided the variable $i$ appears in constraint $j$. The associated edge has weight $a_{ij}$; the number of edges in the graph depends on the sparsity of $A$.

\paragraph{Data.} We consider the very well-known Set Cover and Uncapacitated Facility Location problems with their natural ILP formulations. 
The $1,000$ set cover instances have $50$ subsets and $30$ base elements.
The $1,000$ uncapacitated facility location instances have $10$ facilities and $10$ clients.   The details for randomly generating the instances are detailed in the supplementary material. 

\paragraph{Training.} The experiments were conducted on a Linux machine with a 24-core Intel Xeon Gold 6126 CPU, with 745Gb of RAM, and an NVIDIA Tesla V100-PCIE with 32GB of VRAM. We used Gurobi 12.0.1 \cite{Gurobi} to solve the ILPs, with default cuts, heuristics, and presolve settings turned off. The GNNs were implemented using PyTorch 2.6.0 and Pytorch Geometric 2.6.1.  The details of the implementation are detailed in the supplementary material. 

\subsection{Empirical results}

The GNN is trained using the B\&C tree size vs. gap closed as a proxy. The average tree size obtained on $250$ test instances for each problem after adding the chosen cut from the tableau is reported in Table \ref{tab:results}. The table compares four strategies: the perfect predictor (Optimal) always using the CG tableau cut that results in the smallest B\&C tree size, a classical heuristic selecting a cut according to its parallelism with respect to the objective function (Parallelism, see, e.g., \cite{lodi2009mixed}), a uniform random selection (Random), and the GNN using either the B\&C tree size or the gap closed in training. 
\captionsetup[table]{skip=2pt}
\begin{table}[h]
\footnotesize
\caption{Average tree size on 250 test instances of the GNN trained using either the B\&C tree size or gap closed as a proxy vs two classical benchmarks.}
  \label{tab:results}
  \centering
 \subfloat[Set cover]{\renewcommand{\arraystretch}{1.15} 
  \begin{tabular}{rrr}
    \toprule
    \cmidrule(r){1-3}
    Setting     &  B\&C tree & gap closed  \\ 
    \midrule
    Optimal & 4.95 & 4.95\\
    Parallelism  & 8.29 &  8.29\\ 
    Random  & 9.71 & 9.71\\
    GNN  &  8.27  &  8.65 \\ 
    \bottomrule
  \end{tabular}}
 \subfloat[Facility location]{\renewcommand{\arraystretch}{1.15} 
  \begin{tabular}{rrr}
    \toprule
    \cmidrule(r){1-3}
    Setting     &  B\&C tree & gap closed  \\ 
    \midrule
    Optimal & 86.31 & 86.31 \\
    Parallelism  & 144.09 &  144.09\\ 
    Random  & 152.46 & 152.46\\
    GNN  &  128.85  &  134.61 \\ 
    \bottomrule
  \end{tabular}}
\end{table}

The initial results in Table \ref{tab:results} show that the GNNs are able to learn and provide a solid improvement (facility location) or stay on pair (set cover) with respect to a state-of-the-art cut selection heuristic. The GNN trained by the gap closed score function provides a good proxy, though there is room for improvement for both GNNs with respect to the perfect predictor.

\section{Discussion and open problems}
\label{sec:conclusions}

In this paper, we have presented the first sample complexity lower bounds on the learning-to-cut task and we have empirically analyzed the relationship between two score functions used to assess the quality of a cut. In the sample complexity bounds, no analysis was conducted on the cut candidates that actually close the gap, i.e., cut off the fractional solution. This could give additional information to give better sample complexity bounds in the case of gap closed. Therefore, we conjecture that it is possible to obtain a better upper bound of the sample complexity for the gap closed score because, implicitly, a restricted number of cuts (only those cutting off the fractional solution) are required.



\section{Acknowledgements}

Both authors gratefully acknowledges support from the Office of Naval Research (ONR) grant N00014-24-1-2645.


\bibliographystyle{alpha}

\bibliography{sample}

\appendix
\input{appendix}

\end{document}

%% file: prem.tex
\usepackage[utf8]{inputenc} 
\usepackage[T1]{fontenc}    
\usepackage{hyperref}       
\usepackage{url}            
\usepackage{booktabs}       
\usepackage{amsfonts}       
\usepackage{nicefrac}       
\usepackage{microtype}      
\usepackage{xcolor}         
\usepackage{graphicx}
\usepackage{subcaption}

\usepackage{amsmath,amssymb,amsthm}
\usepackage[capitalize,noabbrev]{cleveref}

\theoremstyle{definition}
\newtheorem{theorem}{Theorem}[section]
\newtheorem{lemma}[theorem]{Lemma}
\newtheorem{corollary}[theorem]{Corollary}
\newtheorem{proposition}[theorem]{Proposition}

\newtheorem{definition}[theorem]{Definition}
\newtheorem{remark}[theorem]{Remark}

\newcommand{\T}{\mathsf{T}} 
\newcommand{\N}{\mathbb{N}}

\newcommand{\F}{\mathcal{F}}

\newcommand{\R}{\mathbb{R}}

\newcommand{\Q}{\mathbb{Q}}
\newcommand{\Z}{\mathbb{Z}}
\newcommand{\x}{\mathbf{x}}

\renewcommand{\u}{\mathbf{u}}

\renewcommand{\b}{\mathbf{b}}
\renewcommand{\c}{\mathbf{c}}

\renewcommand{\O}{\mathcal{O}}

\newcommand{\I}{\mathcal{I}}

\newcommand{\0}{\mathbf{0}}

\DeclareMathOperator{\sgn}{sgn} 

\DeclareMathOperator{\argmax}{arg\,max} 

\DeclareMathOperator{\VCdim}{VCdim} 
\DeclareMathOperator{\Pdim}{Pdim} 

\usepackage{bm}
\usepackage{bbm} 

%% file: introduction.tex
\section{Introduction}

Branch-and-cut algorithms form the cornerstone of integer programming solvers. In recent years, machine learning has been playing a growing role in enhancing those solvers by enabling data-driven decision-making in various components of the algorithm. 
Recent attempts  aim at augmenting those solvers, which often rely on handcrafted heuristics, by training models on  data obtained from solved instances, to predict decisions that lead to faster convergence (which cutting plane -- or cut, for short -- to choose, or which variable to branch on). Specifically referring to cuts, there has been a growing body of work recently. \cite{paulus2022learning}
proposed a neural architecture that employs imitation learning to select cutting planes in mixed-integer linear programs (MILPs). By mimicking a lookahead expert that evaluates the potential impact of cuts on future bounds, their method aims to improve the efficiency of cut selection. In \cite{huang2022learning}, the authors trained a neural network to learn a scoring function evaluating the quality of candidate cuts based on instance-specific features. 
 \cite{tang2020reinforcement} 
 explored the use of deep reinforcement learning to adaptively select cutting planes in integer programming. By formulating cut selection as a Markov Decision Process, their method  trains an agent to make the right cut selection among the Tableaux cuts.  Subsequently, \cite{ling2024learning} 
 addressed the challenge of determining when to stop generating cuts, using reinforcement learning and different features of MILPs to make informed decisions. 
 We refer the reader to the excellent survey \cite{deza2023machine} for a more exhaustive list on previous contributions.

A fundamental question in any learning-based approach for generating cutting planes or making branching decisions during the solving process is how many training samples are needed to ensure good performance across an entire (and potentially unknown) distribution of problem instances. This issue -- referred to as \emph{sample complexity} -- is critical, as it determines the scale of the learning task and directly impacts the feasibility of effectively training  models. Understanding sample complexity helps to overcome some of the inherent challenges in these approaches by providing concrete guidance on how many instances must be solved in order to learn patterns that generalize meaningfully across the distribution.

The motivation for our work stems from the following two observations, leading to two main results.
\begin{enumerate}
    \item \label{ob1} The existing studies applicable to sample complexity of learning-to-cut provide upper bounds for specific learning algorithms, formally referred to as \emph{concept classes}. Those studies are applied to a special family of cutting planes, namely Chv\'atal-Gomory (CG) cuts \cite{Gomory58,Chvatal73}. Specifically, in \cite{balcan2021sample}, the concept class is restricted to functions that return \emph{constant} CG weights applied to any instance. In \cite{cheng2024data}, the CG weights are generated by a neural network taking as input an integer linear program (ILP) instance.\\ 
    \emph{Our contribution} is to provide the first quantitative lower bounds on sample complexity, and study lower bounds that are valid for a wide family of classes. Our lower bounds are discussed in Section \ref{sec:theory} and anticipated in Table \ref{sample-bounds}.

    \item \label{ob2} There are two main scores proposed in the literature to evaluate the quality of a cut. The first one is based on the relative size reduction (or increase) of the branch-and-cut (B\&C) tree size. The second one is the relative improvement in the objective function of the relaxed problem (gap closed, where the gap for a MILP is the relative difference between the value of its linear programming, LP, relaxation and that of its optimal solution). The first score correlates well with the overall running time of the algorithm as it corresponds roughly to the number of LPs solved. However, it is easy to see that it is very expensive to train using the tree size because it requires to solve the problem to optimality to be evaluated. So, the second one could be considered as a proxy of the first, and the natural question we aim at discussing is how good the proxy is both in theory and in practice.\footnote{From the theory side, the upper bounds in \cite{balcan2021sample,cheng2024data} are obtained for the branch-and-cut tree size score, although similar approach would yield the same upper bound for both scores.}\\  
    \emph{Our contribution} is to empirically show the quality of the gap closed proxy and assess the ability of a graph neural network to learn both score functions in practice. Although the gap closed score has been extensively used in the integer programming literature, this is the first principled analysis discussing both scores at the same time both theoretically and computationally. The computational evaluation is conducted in Section \ref{sec:numerical}. 
\end{enumerate}
{
\renewcommand{\arraystretch}{1.2}
\begin{table}[h]
  \caption{Illustration of sample complexity bounds in the case of ReLU neural networks with $W$ weights and $L$ layers, for IP instances with $n$ variables and $m$ constraints, verifying    $M \geq \sum_{i=1}^m \sum_{j=1}^n |A_{ij}| + \sum_{i=1}^m |b_i| $. Here,  $\overline{W} = W - w_1(n+1) m $ where $w_1$ is the number of neurons in the first hidden layer.  The bounds in \textcolor{blue}{blue} are our main theoretical contribution.}
  \label{sample-bounds}
  \centering
  \begin{tabular}{lll}
    \toprule
    \cmidrule(r){1-3}
    Setting     & \multicolumn{2}{c}{B\&C tree or gap closed scores}    \\ 
     & Lower Bound & Upper Bound     \\
    \midrule
   all CG-cuts &  $\textcolor{blue}{\Omega(  \overline{W} L \log(\frac{\overline{W}}{L}) )}$  &   $\mathcal{O}(LW \log (U+m) + W \log  M)$ \\ 
    tableau cuts     & $\textcolor{blue}{\Omega( \overline{W} L \log(\frac{\overline{W}}{L}) )}$   &$\mathcal{O}(LW \log(U+t))$  \\ 
    \bottomrule
  \end{tabular}
\end{table}
}

We would like to point out that our work can be put more broadly in the spectrum of \emph{algorithm selection},  where selecting algorithms based on specific instances is allowed. For example, this is the case of \cite{rice1976algorithm,gupta2016pac} where the
sample complexity of learning mappings from instances to algorithms for particular problems is explored. Our
approach is also related to recent work on algorithm design with predictions, see, e.g., \cite{mitzenmacher2022algorithms} and the references therein.

The remainder of the paper is organized as follows. In Section \ref{sec:preliminaries}, we properly define ILPs and its most successful solution method, i.e., branch and cut, as well as we give the basic definitions of learning theory. In Section \ref{sec:theory}, we discuss our main theoretical result on sample complexity lower bounds. In Section \ref{sec:numerical}, we report on the computational investigation involving the two different score functions to evaluate cut quality. Finally, in Section \ref{sec:conclusions}, we draw some conclusions and outline open research questions.

%% file: appendix.tex
%
%

\section{Proofs of main results}

\begin{proof}[Proof of Theorem \ref{theorem:lowerboundtransfer}]
Theorem \ref{theorem:lowerbound:samplecomplexity} guarantees that $$m_L(\epsilon, \delta) \geq \frac{\fat_{\mathcal{F}_{s, \sigma'}}(\frac{\epsilon}{\alpha})  - 1 }{ \alpha} $$holds for any $ 0 < \epsilon < 1$ and $0 < \delta < 10^{-2}$ and $\epsilon <\frac{1}{4}$.  
    Since $\mathcal{F}$ verifies Assumptions \hyperlink{assumptions1}{1 and 2}, we first use Lemma \ref{transferLemma} and select $\alpha = \epsilon < \frac{1}{4}$ to get
    $$m_L(\epsilon, \delta)  \geq \frac{\VCdim(\mathcal{F}[n])-1}{\epsilon} \geq \frac{\VCdim(\mathcal{F}[n])}{2\epsilon}$$
    where the last inequality holds provided $\VCdim(\mathcal{F}[n]) \geq 2$.
\end{proof}

\begin{lemma}[Transfer Lemma]\label{transferLemma}
    With the same notations and assumptions made on the concept class described in Subsection \ref{subsection:results1} and verifying Assumptions \hyperlink{assumptions1}{1 and 2}, then for every  $ \gamma \in (0,\frac{1}{2}) $ $$  \fat_{\mathcal{F}_{s, \sigma'}}(\gamma) \geq \VCdim(\mathcal{F}[n])   $$	
\end{lemma}

\begin{proof}
Remind that the squeezing function (mapping each coordinate output of the underlying concept class to $[0,1]$) $\sigma'$ verifies $\sigma'((-\infty, 0)) \subset [0, \frac{1}{2})$, $\sigma'([0, +\infty)) \subset [\frac{1}{2}, 1]$ and $(0,1) \subset \sigma'(\R)$. 

Let $0 < \gamma < \frac{1}{2}$ and let $r:= \text{VCdim}(\mathcal{F}[n]) = \max_{i\in[m]} \VCdim(\mathcal{F}_i[n])$. Without loss of generality, we will suppose that the index maximizing this quantity is $i=1$. Therefore, there exists $a_1, \cdots, a_r$ that are shattered by $\sgn(\mathcal{F}[n])$. 
For every labeling $(y_1, \cdots, y_r) \in \{-1,1\}^{r}$, there exists $g \in \mathcal{F}$ such that $(g(a_i))_1 \geq 0$ if $y_i =1$ and $(g(a_i))_1 <  0$ if $y_i  =-1$.

Thanks to Assumption \hyperlink{assumtion1}{1}, we can suppose without loss of generality that the vectors $a_1, \cdots, a_r$  \textbf{do not intersect the positive orthant}. Indeed if they did, we can translate all of them by some vector $\mu$, (in the negative orthant for example), and there are corresponding functions in the concept class to shatter the new vectors. This insures that we can restrict to a list of instances in the positive orthant that do not need to be translated.

We construct $r$ instances, described by linear equalities, to be fat-shattered by $\mathcal{F}_{s, \sigma'}$ with margin $\gamma$ as follows:
\[
\begin{aligned}
  P_{i}:= \{ x \in \mathbb{R}^2: a_i^{t}\, {\bf x} \leq 0, \; 2 x_1 \leq 4, \;  
  2 x_2 \leq  \frac{5}{2}+ 2 \gamma,    \; {\bf{x}} \geq 0 \}
\end{aligned}
\]
\[
\begin{aligned}
       I_i:= \max\{x_1 + x_2 : {\bf x } \in P_i,  \; \x \in \Z^2\}.
\end{aligned}
\]

Our construction of instances can be lifted to $n$ variables and $m$ constraints, simply by adding useless constraints $0 \leq 0$ and keeping the same objective.

Although the constraint $2x_1 \leq 4$ is indeed equivalent to $x_1 \leq 2$, we keep it under that form  as it will be more convenient for our computations.
 First, since the $a_i$'s are not intersecting the positive orthant,  the first constraint is redundant, and we will use the vectors $a_i$ to shatter the instances. For each instance, the objective of the relaxed problem at the optimum is $2+\frac{5}{4}+\gamma$, and one solution is given by $x_1^{*}=2$ and $x_2^{*} = \frac{5}{4}+\gamma$.

Before jumping into the CG cut computations, we first make the following observation: due to the stability under scaling of Assumption \hyperlink{assumtion1}{1}, we can restrict to functions of the concept class such that the first coordinate after applying $\sigma'$ is so small that the associated CG weight $u_1$ (associated to the constraint $a_i^{t}{\bf x} \leq 0$) does not impact the following computations of the CG cuts as we shall see.

In the following, we suppose that $0 \leq u_1 < \frac{1}{2} $ to get rid of the first useless constraint.
Consider the two regions in the $u_2$, $u_3$ space associated to the second and third constraint, giving rise to the  CG-cuts:
 \begin{itemize}
     \item corresponding to the weights $ \frac{1}{2} \leq u_2 \leq 1 - \frac{5}{36}\left(\frac{5}{2} + 2\gamma\right)$ and $ \frac{1}{2} \leq u_3 < \frac{20}{36}$. For each instance, this  yields the inequality:
 $\lfloor 2u_2 \rfloor x_1 + \lfloor 2 u_3 \rfloor x_2 \leq \lfloor 4u_2 + u_3(\frac{5}{2}+ 2\gamma) \rfloor \iff  x_1 +  x_2 \leq 3$ since  $\gamma < \frac{1}{2}$. 
 
 Then the two new vertices of the admissible region are $(2,1)$ and $(2-\gamma, 1 +\gamma)$, for both of them the objective value is $3$, so the amount of  gap closed is $\frac{1}{4}+\gamma$ (the improvement ratio is $\frac{\frac{1}{4}+\gamma}{2+\frac{5}{4}+\gamma} \geq \frac{\gamma}{5}$ since $0< \gamma <1$ (here, the cut actually gives the integral solution).
 
 \item For any $\frac{1}{2} \leq u_2 \leq (1 - \frac{5}{16}) - \frac{\gamma}{4} $ and $0 \leq  \frac{3-(4-\frac{5}{4}-\gamma)}{\frac{5}{2}+2\gamma}   \leq  u_3 < \frac{1}{2}$, the CG-cut associated with $(u_1, u_2, u_3)$ yields the inequality $x_1 \leq 3$: this cut is redundant, the solution is the same as previously  and  the gap closure is $0$.
 \end{itemize}

Hence, we have two CG-cuts that yield for each instance two scores that are  at least $\Omega(\gamma)$  from each other. 

\textit{\textbf{In the case of B\&C-tree size:} Those CG cuts can also be used for the B\&C tree size score: on the one hand, it is clear that the branch-and-cut tree size after adding the first CG-cut to one, as solving the LP only once gives an optimal solution that is integral. On the other hand, adding the redundant cut associated with the second cut at the root gives a branch-and-cut tree size of at least $3$ nodes since one needs to branch at least once on a variable to obtain the integral solution. Therefore we have two CG-cuts that will yield two scores that are at distance $1$ for any of the $n $ instances.}


For any function $\tilde{g}$ in $\mathcal{F}$, $\tilde{g}: \R^{8} \rightarrow \R^{3}$,  we refer to $\tilde{g}$ as $\begin{pmatrix}
           A \\
           b \\ 
           c \\
         \end{pmatrix}  \mapsto \begin{pmatrix}
           g_1(A_1, \cdots) \\
           g_2(A_1,  \cdots) \\
           g_3(A_1, \cdots)
         \end{pmatrix} $ , where $A_1 $ is the first row of $A$. Since the vectors $a_i$ are shattered, for every $i \in [n]$, there exists $\tilde{g}\in \mathcal{F}$ such that $ \tilde{g}(P_i) = \begin{pmatrix}
           g_1(a_i, \cdots) \\
           g_2(a_i, \cdots) \\
           g_3(a_i, \cdots)
         \end{pmatrix}   $ is  equal to $\begin{pmatrix}
           q_i \\
           r_i \\
           \eta_i
         \end{pmatrix}$ for some $\eta_i \geq 0$ if  $y_i =1$, and $\begin{pmatrix}
           q'_i \\
            r'_i \\ 
            -\eta_2 
         \end{pmatrix} $ for some $   \eta_2 > 0  $ if $y_i = -1$.  

Here we implicitely used Assumption \hyperlink{assumption1}{2}, by supposing that the VC dimension of one of the coordinates of the functions in $\mathcal{F}$, when restricted to the first $n$ entries, does not depend on the choice of the coordinate, nor other entries of the instance. Using again the Assumption \hyperlink{assumption1}{1} of closure under scaling, we rescale the second component while keeping the other untouched, so that $q_i $ and $q'_i$ are small (mentioned before the computation of the CG weights), and such that all $r_i$ verify $ \sigma'(r_i) \in [\frac{1}{2}, 1 - \frac{5}{36}(\frac{5}{2} + 2\gamma)]$ and all $\sigma'(r'_i) \in [\frac{1}{2}, 1 - \frac{5}{16}]$, i.e. the right intervals corresponding to the CG weights computed previously. This is possible as $(0,1) \subset \sigma'(\R)$ by assumption.

          Also, since $\sigma'((-\infty, 0)) \subseteq [0, \frac{1}{2})$ and $\sigma'([0, +\infty)) \subseteq [\frac{1}{2},1]$, this implies that when $y_i=-1$, the weights obtained after applying the squeezing function  generates the CG-cut ${\bf u}_1$, and when $y_i = -1$,  the weights generating the CG-cut ${\bf u}_2$.

          Therefore, the instances $P_1, \cdots, P_n$ with $n = \VCdim(\mathcal{F})$ are $\gamma$-fat shattered (with a witness that depends on the score considered), so $\fat_{\mathcal{F}_{s, \sigma'}}(\gamma) \geq  \VCdim(\mathcal{F}[n]) = \max_{i\in[m]} \VCdim(\mathcal{F}_i[n])$.
\end{proof}

\begin{definition}[Pseudo-dimension]\label{def:pseudo-dimension}
    Let $\F$ be a non-empty collection of functions from an input space $\I$ to $\R$. 
    For any positive integer $t$, we say that a set $\{I_1, ..., I_t\} \subseteq \I$ is pseudo-shattered by $\F$ if there exist real numbers $s_1, \dots, s_t$ such that 
    \begin{equation*}\label{eq:def:pseudodimension}
        2^t = \left|\left\{ \left( \sgn(f(I_1)-s_1),\dots, \sgn(f(I_t)-s_t) \right) : f \in \F \right\}\right|. 
    \end{equation*}
    The \emph{pseudo-dimension of $\F$}, denoted as $\Pdim(\F)  \in \mathbb{N} \cup \{+\infty\}$, is the size of the largest set that can be pseudo-shattered by $\F$. 
    \end{definition}

\begin{proof}[Proof of Corollary \ref{corollary:learningcomplexity}]
    We use here the following fact (see Definition above) that for any concept class $\mathcal{F}$ with real outputs, $\text{Pdim}(\mathcal{F}) \geq \fat_{F}(\gamma)$ for all $\gamma > 0$, see for example  \cite[Theorem 11.13]{anthony2009neural}.
    
    Also, by assumption each  $\mathcal{F}_i[n]$ is closed under translation of the input so if a set of vectors $x_1, \cdots, x_n$ are pseudo-shattered, there are also shattered, leading to 
    $$\VCdim (\mathcal{F}_i[n]) \geq \Pdim( \mathcal{F}_i[n] ) \geq \fat_{\mathcal{F}_i[n]}(\gamma)$$
    proving the claim.
\end{proof}

\begin{proof}[Proof of Proposition \ref{lowerbound:neuralnetworks}]
 This is a direct application of Theorem \ref{theorem:lowerboundtransfer} combined with state-of-the art VC-dimension lower bound for ReLU neural networks \cite[Theorem 3]{bartlett2019nearly}. 
\end{proof}

\begin{proof}[Proof of Theorem \ref{theorem:lowerboundtransfertableau}]
    We perform a similar reasoning as in the proof of Theorem \ref{theorem:lowerbound:samplecomplexity} with the following ingredients (and similar notations):

    \begin{itemize}
    \item  we use Lemma \ref{lemma:transfertableau} and justify the lifting to more variables.
    \item We do not have to perform the CG weights computation as previously, only to make sure that we can shatter the vector $a_i$ of the redundant constraint $a_i^{t}{\bf x } \leq 0 $ added to the instance of Lemma \ref{lemma:transfertableau}, in such a way that the score after is maximized for the index of the row corresponding to one or the other cut. 
    \end{itemize}
    This allows us to shatter a collection of instances in order to choose one or the other CG-cut, given the arbitrary set of labels $y_1, \cdots, y_r$.
\end{proof}

\begin{lemma}\label{lemma:transfertableau}
There exists an two-variable ILP instance with two constraints such that the two cuts from the tableau have both scores at distance $> \frac{1}{5}$. 

\end{lemma}
\begin{lemma}
Such instance is the same as the one illustrated in Section \ref{section:branchandcut} and illustrated in Figure \ref{figure:polytope1}
    \begin{align*}
\max \quad & 5x_1 + 8x_2 \\
\text{subject to} \quad & x_1 + x_2 \leq 6 \\
& 5x_1 + 9x_2 \leq 45 \\
& x_1, x_2 \geq 0 \\
& x_1, x_2 \in \mathbb{Z}
\end{align*}

\end{lemma}

\begin{figure}[h!]
    \centering
    \includegraphics[height=6cm]{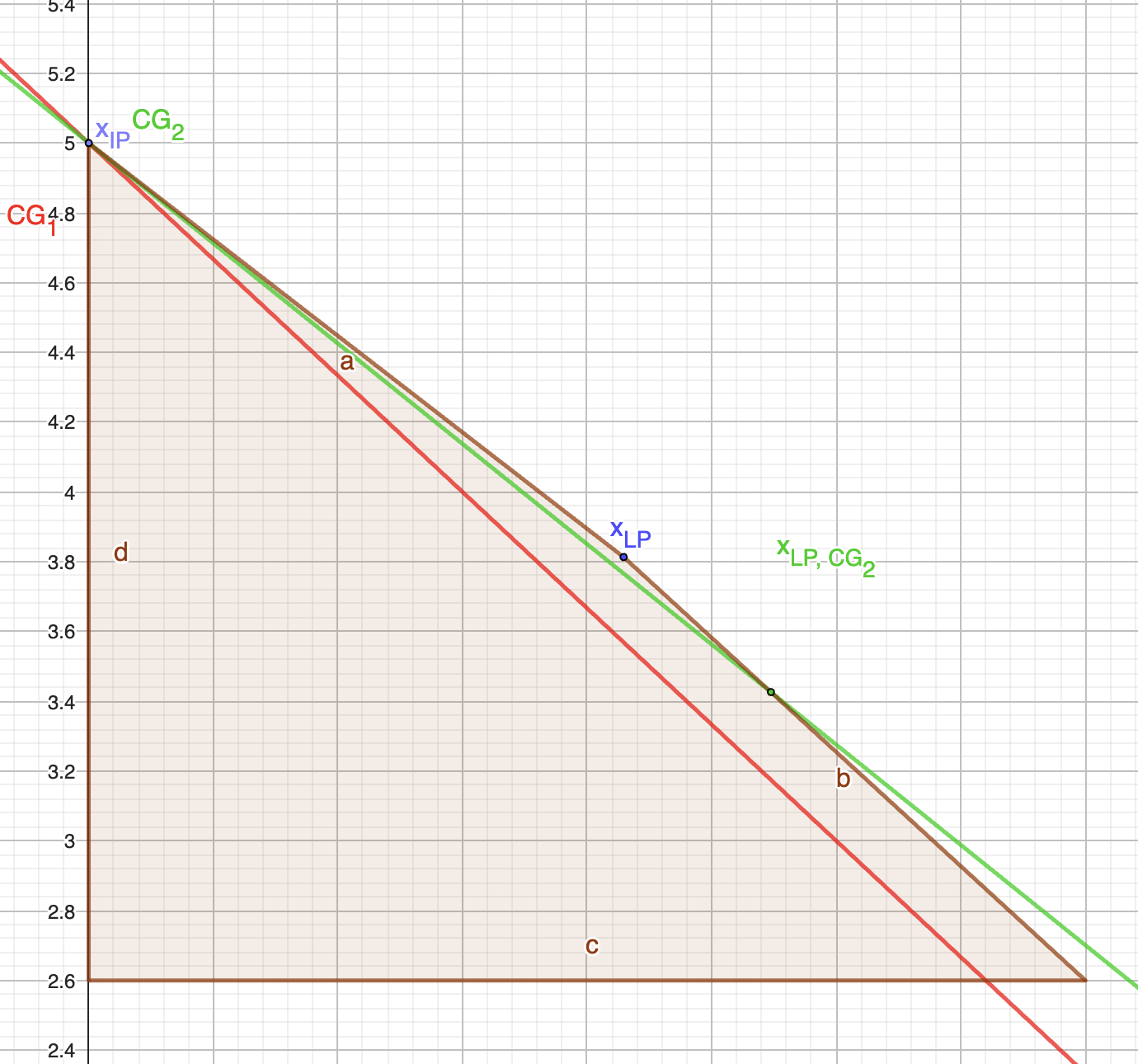}
    \caption{Example of $2D$ instance used to construct the lower-bound and prove Lemma \ref{lemma:transfertableau}. Both cuts CG1 and CG2 are possible cuts from the Optimal Tableau. The cut derived from CG1 gives directly the optimum $(0,5)$, whereas the cut CG2 gives a suboptimal fractional solution. Both cuts are then mapped via the redundant constraint of the collection of instances, to $\gamma$-shatter the instances according to the score considered.}
    \label{figure:polytope1}
\end{figure}

\begin{proof}[Proof of Proposition \ref{theorem:lowerboundtransfertableau}]
     This is a direct application of Theorem \ref{theorem:lowerboundtransfertableau} combined with state-of-the art VC-dimension lower bound for ReLU neural networks \cite[Theorem 3]{bartlett2019nearly}. 
\end{proof}

%% file: main.bbl
\newcommand{\etalchar}[1]{$^{#1}$}
\begin{thebibliography}{HWL{\etalchar{+}}22}

\bibitem[AB09]{anthony2009neural}
Martin Anthony and Peter~L Bartlett.
\newblock {\em Neural network learning: Theoretical foundations}.
\newblock cambridge university press, 2009.

\bibitem[BHLM19]{bartlett2019nearly}
Peter~L Bartlett, Nick Harvey, Christopher Liaw, and Abbas Mehrabian.
\newblock Nearly-tight vc-dimension and pseudodimension bounds for piecewise
  linear neural networks.
\newblock {\em Journal of Machine Learning Research}, 20(63):1--17, 2019.

\bibitem[BPSV21]{balcan2021sample}
Maria-Florina~F Balcan, Siddharth Prasad, Tuomas Sandholm, and Ellen Vitercik.
\newblock Sample complexity of tree search configuration: Cutting planes and
  beyond.
\newblock {\em Advances in Neural Information Processing Systems},
  34:4015--4027, 2021.

\bibitem[CCZ14]{conforti2014integer}
Michele Conforti, G{\'e}rard Cornu{\'e}jols, and Giacomo Zambelli.
\newblock {\em Integer programming}, volume 271.
\newblock Springer, 2014.

\bibitem[Chv73]{Chvatal73}
V.~Chv\'atal.
\newblock Edmonds polytopes and a hierarchy of combinatorial problems.
\newblock {\em Discrete Mathematics}, 4:335--337, 1973.

\bibitem[CKFB24]{cheng2024data}
Hongyu Cheng, Sammy Khalife, Barbara Fiedorowicz, and Amitabh Basu.
\newblock Data-driven algorithm design using neural networks with applications
  to branch-and-cut.
\newblock {\em arXiv preprint arXiv:2402.02328}, 2024.

\bibitem[CLC{\etalchar{+}}24]{chen2024rethinking}
Ziang Chen, Jialin Liu, Xiaohan Chen, Wang Wang, and Wotao Yin.
\newblock Rethinking the capacity of graph neural networks for branching
  strategy.
\newblock {\em Advances in Neural Information Processing Systems},
  37:123991--124024, 2024.

\bibitem[DK23]{deza2023machine}
Arnaud Deza and Elias~B Khalil.
\newblock Machine learning for cutting planes in integer programming: A survey.
\newblock {\em arXiv preprint arXiv:2302.09166}, 2023.

\bibitem[Fel16]{feldman2016generalization}
Vitaly Feldman.
\newblock Generalization of erm in stochastic convex optimization: The
  dimension strikes back.
\newblock {\em Advances in Neural Information Processing Systems}, 29, 2016.

\bibitem[Gom58]{Gomory58}
R.E. Gomory.
\newblock Outline of an algorithm for integer solutions to linear programs.
\newblock {\em Bulletin of the American Mathematical Society}, 64:275--278,
  1958.

\bibitem[GR16]{gupta2016pac}
Rishi Gupta and Tim Roughgarden.
\newblock A pac approach to application-specific algorithm selection.
\newblock In {\em Proceedings of the 2016 ACM Conference on Innovations in
  Theoretical Computer Science}, pages 123--134, 2016.

\bibitem[Gur]{Gurobi}
{\em Gurobi Optimization, LLC}.

\bibitem[HWL{\etalchar{+}}22]{huang2022learning}
Zeren Huang, Kerong Wang, Furui Liu, Hui-Ling Zhen, Weinan Zhang, Mingxuan
  Yuan, Jianye Hao, Yong Yu, and Jun Wang.
\newblock Learning to select cuts for efficient mixed-integer programming.
\newblock {\em Pattern Recognition}, 123:108353, 2022.

\bibitem[Lod09]{lodi2009mixed}
Andrea Lodi.
\newblock Mixed integer programming computation.
\newblock In {\em 50 years of integer programming 1958-2008: From the early
  years to the state-of-the-art}, pages 619--645. Springer Berlin Heidelberg
  Berlin, Heidelberg, 2009.

\bibitem[LWW24]{ling2024learning}
Haotian Ling, Zhihai Wang, and Jie Wang.
\newblock Learning to stop cut generation for efficient mixed-integer linear
  programming.
\newblock In {\em Proceedings of the AAAI Conference on Artificial
  Intelligence}, volume~38, pages 20759--20767, 2024.

\bibitem[MV22]{mitzenmacher2022algorithms}
Michael Mitzenmacher and Sergei Vassilvitskii.
\newblock Algorithms with predictions.
\newblock {\em Communications of the ACM}, 65(7):33--35, 2022.

\bibitem[NW88]{nemhauser1988integer}
George~L Nemhauser and Laurence~A Wolsey.
\newblock {\em Integer and combinatorial optimization}, volume~18.
\newblock Wiley New York, 1988.

\bibitem[PZK{\etalchar{+}}22]{paulus2022learning}
Max~B Paulus, Giulia Zarpellon, Andreas Krause, Laurent Charlin, and Chris
  Maddison.
\newblock Learning to cut by looking ahead: Cutting plane selection via
  imitation learning.
\newblock In {\em International conference on machine learning}, pages
  17584--17600. PMLR, 2022.

\bibitem[Ric76]{rice1976algorithm}
John~R Rice.
\newblock The algorithm selection problem.
\newblock In {\em Advances in computers}, volume~15, pages 65--118. Elsevier,
  1976.

\bibitem[Sch86]{sch}
Alexander Schrijver.
\newblock {\em Theory of Linear and Integer Programming}.
\newblock John Wiley and Sons, New York, 1986.

\bibitem[SSSSS09]{shalev2009stochastic}
Shai Shalev-Shwartz, Ohad Shamir, Nathan Srebro, and Karthik Sridharan.
\newblock Stochastic convex optimization.
\newblock In {\em COLT}, volume~2, page~5, 2009.

\bibitem[TAF20]{tang2020reinforcement}
Yunhao Tang, Shipra Agrawal, and Yuri Faenza.
\newblock Reinforcement learning for integer programming: Learning to cut.
\newblock In {\em International conference on machine learning}, pages
  9367--9376. PMLR, 2020.

\end{thebibliography}
